\documentclass[12pt]{amsart}
\usepackage[margin=3cm]{geometry}

\usepackage{graphicx} 
\usepackage{tikz}
\usepackage{amsmath}
\usepackage{amsthm}
\usepackage{mathtools}
\usepackage{amssymb}
\usepackage{mathrsfs}
\def\ed{\mathrm{d}}
\def\mc#1{\mathcal{#1}}
\usepackage{color}
\usepackage{hyperref}
\usepackage{xcolor}
\hypersetup{
    colorlinks=true,
    linkcolor=blue,  
    urlcolor=cyan,
    pdftitle={},
    pdfauthor={},
}
\newtheorem{theorem}{Theorem}
\newtheorem{lemma}{Lemma}
\newtheorem{proposition}{Proposition}

\newtheorem{remark}{Remark}

\DeclareMathOperator*{\argmin}{argmin}

\title[Stability of Wasserstein projections in convex order]{Stability of Wasserstein projections in convex order via metric extrapolation}

\author{Jakwang Kim}
\address{Jakwang Kim, School of Data Science, The Chinese University of Hong Kong, Shenzhen, Guangdong, 518172, P.R. China.}
\email{jakwangkim@cuhk.edu.cn}
\author{Young-Heon Kim}
\address{Young-Heon Kim, Department of Mathematics, University of British Columbia, 1984 Mathematics Road, Vancouver, British Columbia, V6T 1Z2, Canada.}
\email{yhkim@math.ubc.ca}
\author{Andrea Natale}
\address{Andrea Natale, Université Paris-Saclay, CNRS, Inria, Laboratoire de mathématiques d’Orsay, ParMA, 91405, Orsay, France.}
\email{andrea.natale@inria.fr}

\date{\today}


\begin{document}

\thanks{JK is supported by CUHK-SZ start-up UDF03004229.
	YHK is partially supported by the Natural Sciences and Engineering Research Council of Canada (NSERC), with Discovery Grant RGPIN-2019-03926 and RGPIN-2025-06747, as well as Exploration Grant (NFRFE-2019-00944) from the New Frontiers in Research Fund (NFRF). YHK is also a member of the Kantorovich  Initiative (KI) which is supported by the PIMS Research Network
	(PRN) program of the Pacific Institute for the Mathematical Sciences  (PIMS). We thank PIMS for their generous support;  report identifier PIMS-20250725-PRN01. Part of this work was completed during YHK’s visit at the Korea Advanced Institute of Science and Technology  (KAIST), and we thank them for their hospitality and the excellent environment. AN acknowledges funding by the Agence Nationale de la Recherche (ANR), project ANR-25-CE40-3242-01.}

\begin{abstract}
We build on recent work linking backward and forward $W_2$-projections in convex order with the recently introduced metric extrapolation problem to derive new quantitative stability estimates for both problems.
\end{abstract}

\maketitle

\section{Introduction}
Given two probability measures with finite second moments $\mu,\nu \in \mc{P}_2(\mathbb{R}^d)$, we say that $\mu$ is dominated by $\nu$ in convex order, denoted $\mu \preceq \nu$,  if for any convex function $\varphi:\mathbb{R}^d \rightarrow \mathbb{R}$,
\[
    \int_{\mathbb{R}^d} \varphi\, \ed \mu \leq \int_{\mathbb{R}^d} \varphi \,\ed \nu\,.
\]
In this work we consider the backward and forward 2-Wasserstein projections in convex order,  originally introduced in \cite{alfonsi2020sampling} and defined as follows:
\[
\text{(backward)}~ \mathscr{P}^{cx}_{\preceq \nu} (\mu) \coloneqq  \argmin_{\rho\preceq \nu} W_2^2(\rho,\mu) \,,\quad \text{(forward)}~ \mathscr{P}^{cx}_{\mu\preceq} (\nu) \coloneqq  \argmin_{\mu\preceq \rho} W_2^2(\rho,\nu)\,.
\]
{Here, the 2-Wasserstein distance $W_2(\rho,\nu)$ is defined as 
\begin{align*}
    W_2^2(\rho,\nu )= \inf_{\pi \in \Pi(\rho, \nu)} \int_{\mathbb{R}^d \times \mathbb{R}^d} |x-y|^2 \ed\pi (x, y)
\end{align*}
where $\Pi(\rho, \nu)$ denotes the set of couplings between $\rho$ and $\nu$.
}

The backward projection problem always admits a unique solution so  $\mathscr{P}^{cx}_{\preceq \nu} (\mu)$ is always well-defined \cite{alfonsi2020sampling,gozlan2020mixture,kim2024backward}. In contrast, the forward problem admits a unique solution  if $d=1$ or as long as the measure $\nu$ being projected is absolutely continuous \cite{alfonsi2020sampling,kim2024backward}.

In this work we study the stability of $\mathscr{P}^{cx}_{\preceq \nu} (\mu)$ and $\mathscr{P}^{cx}_{\mu\preceq} (\nu)$ with respect to $\mu$ and $\nu$ in arbitrary dimensions. The main idea is to exploit a recently discovered duality between the backward projection problem and the so-called metric extrapolation problem. 

\subsection{Metric extrapolation} The metric extrapolation problem is a variational problem introduced in \cite{gallouet2024geodesic,Gallouet2025} to define extensions of constant-speed length-minimizing geodesics in the Wasserstein space (referred to simply as  minimizing geodesics in the following). These are curves $\gamma: [0,t] \rightarrow \mc{P}_2(\mathbb{R}^d)$ satisfying
\[
W_2(\gamma(s_1),\gamma(s_2)) = \frac{|s_2-s_1|}{t} W_2(\gamma(0),\gamma(t))
\]
for all $s_1,s_2\in[0,t]$. While for any $\nu_0, \nu_1\in \mc{P}_2(\mathbb{R}^d)$ there always exists a minimizing geodesic on $[0,1]$ such that $\gamma(0)=\nu_0$ and $\gamma(1) = \nu_1$,
there might not be any minimizing geodesic with the same property defined on a larger interval $[0,t]$, with $t>1$.

The metric extrapolation of a geodesic from $\nu_0$ to $\nu_1$ at any time $t>1$ is however always well-defined as the unique minimizer of the following functional:
\begin{equation}\label{eq:extrapolation}
    \mathscr{E}^t(\nu_0,\nu_1)  \coloneqq \argmin_\rho \left\{ \frac{W_2^2(\rho,\nu_1)}{2(t-1)} - \frac{W_2^2(\rho,\nu_0)}{2t} \right\}\,.
\end{equation}
By elementary inequalities, if there exists a minimizing geodesic $\gamma: [0,t] \rightarrow \mc{P}_2(\mathbb{R}^d)$ such that $\gamma(0)=\nu_0$ and $\gamma(1) = \nu_1$ then necessarily $\mathscr{E}^t(\nu_0,\nu_1) = \gamma(t)$. When this is not the case however, \cite{Gallouet2025} found that there still exists a minimizing geodesic $\tilde{\gamma}:[0,t] \rightarrow \mc{P}_2(\mathbb{R}^d)$ such that $\tilde{\gamma}(1) = \nu_1$ and $\tilde{\gamma}(t) = \mathscr{E}^t(\nu_0,\nu_1) $, but $\tilde{\gamma}(0) \neq \nu_0$ is precisely $\mathscr{P}_{\preceq \nu_0}^{cx}(\mu)$, where $\mu$ is a dilated version of $\nu_1$ (see Theorem \ref{th:projextra} for a precise statement).

\subsection{Main results} In this work we review and make explicit the correspondence between metric extrapolation and the (backward and forward) $W_2$-projections in convex order (Theorem \ref{th:projextra}), by assembling recent results established in \cite{Gallouet2025,bourne2025semi,alfonsi2020sampling,gozlan2020mixture,kim2024backward}. We use such correspondence to establish quantitative stability estimates for both problems.

First, exploiting the fact that the backward projection problem with $\nu$ fixed can be lifted to a projection on a Hilbert space, we prove that the map $\mu \mapsto \mathscr{P}_{\preceq \nu}^{cx} (\mu)$ is non-expansive { (Lemma~\ref{lem:nonexpansive}).}
In turn, we show that this fact implies a Lipschitz estimate also on the map $\nu_1 \mapsto \mathscr{E}^t(\nu_0,\nu_1)$ (Proposition \ref{prop:extralip}). We further show that the map $\nu \mapsto \mathscr{P}_{\preceq \nu}^{cx} (\mu)$ is 1/2-H\"older (Proposition \ref{prop:backwardstab}).

 { While finalizing our manuscript, we found that Lemma~\ref{lem:nonexpansive} and Proposition \ref{prop:backwardstab}
 had also been independently established in a recent work \cite[Proposition 2.2, Proposition 2.9]{alfonsi2025wasserstein} by Alfonsi and Jourdain; Lemma~\ref{lem:nonexpansive} by a similar method and Proposition \ref{prop:backwardstab} by a different method.  Their work further includes a detailed study of the special case where both  $\mu$ and $\nu$ are Gaussians.}

{  Our approach is novel, for in particular, it}  can be directly applied to the forward projection problem to show that, upon some regularity assumption, { the map} $(\mu,\nu) \mapsto \mathscr{P}_{\mu \preceq }^{cx} (\nu)$ is also H\"older continuous (Proposition \ref{prop:forwardstab}). { For this we use a recent stability result on Wasserstein barycenters  \cite[Theorem 1.5]{carlier2024quantitative}.} 

Note that even if the latter results may not be sharp, the fundamental idea of this work is that in view of Theorem \ref{th:projextra}, stability results on the extrapolation problem can be transferred to the projection problems and vice-versa.

\section{Projections in convex order and metric extrapolation}
In this section, we recall the equivalence between the $W_2$-projections in convex order and the metric extrapolation problem. This will be central to establish the stability of the backward and forward projections.

We start by recalling that if $\nu_0 \in \mc{P}_2^{ac}(\mathbb{R}^d)$ is absolutely continuous, then for any $\nu_1 \in \mc{P}_2(\mathbb{R}^d)$ there exists a uniquely defined minimizing geodesic $\gamma:[0,1]\rightarrow \mc{P}_2(\mathbb{R}^d)$ such that  $\gamma(0)=\nu_0$ and $\gamma(1) =\nu_1$. This is given by $\gamma(s) = ( (1-s)\mathrm{Id} + s \nabla u)_\# \nu_0$ for all $s\in[0,1]$, where $\#$ denotes the push-forward of measures and $\nabla u$ is the gradient of a convex function providing the unique optimal transport map from $\nu_0$ to $\nu_1$, i.e., $\pi =(\mathrm{Id}, \nabla u)_\# \nu_0$ is the unique optimal coupling between $\nu_0$ and $\nu_1$.

For any $\lambda>0$ we denote by $D^\lambda :\mathbb{R}^d\rightarrow \mathbb{R}^d$ the dilation map defined by
\[
D^\lambda(x) \coloneqq \lambda x\,.
\]
The following theorem summarizes some of the results from \cite{Gallouet2025}, and combines them with those from \cite{alfonsi2020sampling,gozlan2020mixture,kim2024backward}.  From now on, for a given $t>1$, we always set $\theta\coloneqq (t-1)/t$, omitting the dependency on $t$ to keep the notation light.
\begin{theorem}\label{th:projextra} Let $t>1$, 
$\theta \coloneqq (t-1)/t$, and let $\nu_0,\nu_1\in \mc{P}_2(\mathbb{R}^d)$  be arbitrary measures. The following holds:
\begin{enumerate}
    \item There exists a unique minimizing geodesic $\gamma:[0,t] \rightarrow \mc{P}_2(\mathbb{R}^d)$ such that \[
    \gamma(1) = \nu_1~\text{ and }~\gamma(t) = \mathscr{E}^t(\nu_0,\nu_1) .\]
    Moreover $\gamma(0) = \mathscr{P}^{cx}_{\preceq \nu_0}(D^{1/\theta}_\# \nu_1)$  and there exists a lower semi-continuous $\theta$-strongly convex potential $u:\mathbb{R}^d \rightarrow \mathbb{R} \cup \{+\infty\}$ such that $\gamma(s) = (s\mathrm{Id}+ (1-s)\nabla u^*)_\#\nu_1$ for all $s\in[0,t]$. 
    { Here $u^*$ is the Legendre transform of $u$.}

\item If  $\nu_0$ is absolutely continuous, there exists a unique minimizing geodesic $\tilde{\gamma}:[0,t]\rightarrow \mc{P}_2(\mathbb{R}^d)$ such that
\[
\tilde{\gamma}(0)=\nu_0 \quad \text{and}\quad \tilde{\gamma}(t) = \mathscr{E}^t(\nu_0,\nu_1)\,.
\]
Moreover,  $\tilde{\gamma}(1) = \mathscr{P}^{cx}_{\nu_1 \preceq}(D^\theta_\#\nu_0)$ and $\tilde{\gamma}(s) = ((1-s)\mathrm{Id}+ s\nabla u)_\#\nu_0$ for all $s\in[0,t]$, with the same potential $u$ as in point $(1)$.    
\end{enumerate}
\end{theorem}

\begin{proof} For the first point, observe that by \cite[Theorem 2.6]{Gallouet2025} there exists a unique minimising geodesic $\gamma$ as in the statement, and we only need to check that $\gamma(0) = \nabla u^*_\# \nu_1$ is equal to $\mathscr{P}^{cx}_{\preceq \nu_0} (D^{1/\theta}_\# \nu_1)$. By the same theorem, the potential $u$ defining the geodesic solves
	\begin{equation}\label{eq:toland}
	\inf\left\{ \int u \ed \nu_0 + \int u^* \ed \nu_1 ~:~ u -\theta\frac{|\cdot|^2}{2} \text{ is convex and l.s.c.} \right\}\,.
	\end{equation}
Setting $v\coloneqq  u^*( \theta\,\cdot)/\theta$, we have $v^* = u/\theta$. Moreover, since $u$ is $\theta$-convex, $v$ is $C^1$ convex and $\nabla v$ is 1-Lipschitz. Conversely, for any such $v$, $u = \theta v^*$ is $\theta$-convex and l.s.c.. Hence $u$ solves \eqref{eq:toland} if and only if $v$ solves
	\[
\inf\left\{ \int  v^* \ed \nu_0 + \int  v \ed D^{1/\theta}_\# \nu_1 ~:~ v\in C^1(\mathbb{R}^d) \text{ is convex and } \nabla v \text{ is 1-Lipschitz}  \right\}\,.
\]
By Theorem 1.1 in \cite{gozlan2020mixture}, this is precisely the dual formulation of the backward projection problem. Moreover, by Theorem 1.2 in \cite{gozlan2020mixture}, we have that 
\begin{equation}\label{eq:projv}
\nabla v_\# (D^{1/\theta}_\# \nu_1) = \mathscr{P}^{cx}_{\preceq \nu_0}(D^{1/\theta}_\#\nu_1).
\end{equation} By direct calculation, $\nabla v = (\nabla u^*) \circ D^\theta$. Combining this with \eqref{eq:projv}, we find precisely that $\nabla u^*_\# \nu_1 =  \mathscr{P}^{cx}_{\preceq \nu_0}(D^{1/\theta}_\# \nu_1)$, and we are done.

{ Note that the first point of the theorem can also be derived  as a direct consequence of \cite[Theorem 3.3]{Gallouet2025} stating the equivalence (up to a scaling factor) between the metric extrapolation problem and the weak optimal transport characterization of the backward projection (see \cite[Proposition 1.1]{gozlan2020mixture}} or \cite[Theorem 2.1]{alfonsi2020sampling}); see also Remark 4.4 in  \cite{bourne2025semi} for a direct argument for the equivalence of the two problems in a semi-discrete setting.  

To prove the second point, observe that 
by Theorem 8.3 and 8.5 in \cite{kim2024backward}, the Legendre dual of $v$, that is, $v^{*} = u/\theta$ is the optimal potential (i.e., $ \nabla u/\theta$ is the optimal transport map) from $\nu_0$ to $\mathscr{P}^{cx}_{\mu\preceq}(\nu_0)$, where $\mu \coloneqq D_\#^{1/\theta} \nu_1$.  In particular, 
	\[\nabla u_\# \nu_0 = D^\theta_\#\mathscr{P}^{cx}_{\mu\preceq}(\nu_0) = \mathscr{P}^{cx}_{\nu_1\preceq}(D^\theta_\# \nu_0) ,\]
where the second equality follows from observing that $D^\theta_\#\mathscr{P}^{cx}_{\mu\preceq}(\nu_0)$ minimizes \[\sigma \mapsto W_2(D^{1/\theta}_\#\sigma,\nu_0) = \frac{1}{\theta} W_2(\sigma,D^{\theta}_\#\nu_0) ,\] under the constraint $\mu \preceq D^{1/\theta}_\# \sigma$, or equivalently $\nu_1 \preceq \sigma$. On the other hand, as a consequence of Theorem 2.6 in \cite{Gallouet2025} (see in particular equation (2.26) therein),
\[
((1-t)\mathrm{Id} + t  \nabla u)_\# \nu_0 = \mathscr{E}^t(\nu_0,\nu_1).
\]
Finally, since $u$ is $\theta$-strongly convex and $\nu_0$ is absolutely continuous, $\tilde{\gamma}$, 
given by
\[
\tilde{\gamma}(s) = ((1-s)\mathrm{Id} + s  \nabla u)_\# \nu_0, \quad s\in [0,t],
\]
is the unique minimizing geodesic connecting its endpoints.
\end{proof}

\section{Non-expansiveness of the backward projection in convex order}

In this section we prove that for any $\nu\in\mc{P}_2(\mathbb{R}^d)$, the map $\mathscr{P}_{\preceq \nu}^{cx}(\cdot)$ is non-expansive. This will also imply a new Lipschitz stability estimate on the metric extrapolation.

For a given absolutely continuous reference measure $\rho_0 \in \mc{P}_2^{ac}(\mathbb{R}^d)$, consider the set
\begin{equation}\label{eq:constraintX}
\{ X\in L^2(\rho_0) \,:\, X_\#\rho_0 \preceq \nu\}\,.
\end{equation}
This is a convex and bounded set.
In fact, given $X_0$ and $X_1$ in this set,  denoting $X_s = (1-s) X_0 +s X_1$ for any $s\in[0,1]$, we have that for any convex function $\varphi: \mathbb{R}^d \rightarrow \mathbb{R}$
\[
\int\varphi\, \ed (X_s)_\# \rho_0 = \int\varphi(X_s) \,\ed \rho_0 \leq (1-s)  \int \varphi(X_0) \, \ed\rho_0 + s  \int \varphi(X_1) \, \ed \rho_0 \leq \int \varphi \, \ed \nu\,,
\]
and so $(X_s)_\# \rho_0 \preceq \nu$.

It is also bounded, since choosing as test function $\varphi : x \rightarrow |x|^2$, we obtain
\[
\|X\|^2_{L^2(\rho_0)} \leq \int_{\mathbb{R}^d}|x|^2 \ed \nu(x)\,.
\]

Fix $\mu \in \mc{P}_2 (\mathbb{R}^d)$ and consider a transport map $X_\mu \in L^2(\rho_0)$ (not necessarily optimal) with $(X_\mu)_\# \rho_0 = \mu$, and the problem
\begin{equation}\label{eq:projl2}
\inf_{X\in L^2(\rho_0)} \left\{ \frac{1}{2}\|X - X_\mu \|^2_{ L^2(\rho_0)} \; :\; X_\# \rho_0  \preceq \nu \right\}\,.
\end{equation}
We now show that this problem is equivalent to the backward projection problem. For this, we recall that for any $\mu,\nu \in \mc{P}_2(\mathbb{R}^d)$ there always exists a unique (Lipschitz continuous) optimal transport map, say $T$,  from $\mu$ to $\mathscr{P}^{cx}_{\preceq \nu}(\mu)$. This can be deduced from Theorem \ref{th:projextra}, but is proven in \cite{alfonsi2020sampling,gozlan2020mixture} for example.

\begin{lemma}\label{lem:lift} 
{For  $\rho_0 \in \mc{P}_2^{ac}(\mathbb{R}^d)$ and $\mu, \nu  \in \mc{P}_2(\mathbb{R}^d)$,}
Problem \eqref{eq:projl2} admits a unique solution $X^*\in L^2(\rho_0)$. Furthermore, $X^* = T\circ X_\mu$ where $T\in L^2(\mu)$ is the unique optimal transport map from $\mu$ to $\mu^* \coloneqq \mathscr{P}_{\preceq \nu}^{cx}(\mu)$. In particular
\begin{equation}\label{eq:equality}
X^*_\# \rho_0 =  \mu^* \quad \text{and} \quad W_2(\mu,\mu^*) = \|X_\mu - X^*\|_{L^2(\rho_0)}\,. 
\end{equation}
\end{lemma}
\begin{proof} 
On the one hand, for any admissible $X$, since $(X,X_\mu)_\# \rho_0\in \Pi(X_\#\rho_0,\mu)$ is an admissible coupling between $X_\# \rho_0$ and $\mu$,  we have
\[
W_2(\mu,\mu^*) \leq W_2(\mu, X_\# \rho_0) \leq \|X - X_\mu\|_{L^2(\rho_0)}.
\]
On the other hand,
\[
    W_2(\mu,\mu^*)  = \|T-\mathrm{Id}\|_{L^2(\mu)} = \|T \circ X_\mu-X_\mu \|_{L^2(\rho_0)}\geq \inf_{X \in L^2(\rho_0) : X_\#\rho_0 \preceq \nu}\|X - X_\mu\|_{L^2(\rho_0)}\,.
\]
This shows that $X^* = T\circ X_\mu$ is a solution, and it satisfies \eqref{eq:equality}. Its uniqueness follows from the strong convexity of the problem.
\end{proof}

Now, given another arbitrary measure $\tilde{\mu}$ and an arbitrary map $X_{\tilde{\mu}}\in L^2(\rho_0)$ such that $(X_{\tilde \mu})_\# \rho_0= \tilde{\mu}$, denoting by $\tilde{X}^*$ the  solution to problem \eqref{eq:projl2} with $\tilde{\mu}$ replacing $\mu$, we have
\begin{equation}\label{eq:stabproj}
W_2(\mathscr{P}_{\preceq \nu}^{cx} (\mu),\mathscr{P}_{\preceq \nu}^{cx} (\tilde{\mu}))\leq \|X^* - \tilde{X}^*\|_{L^2(\rho_0)} \leq \| X_\mu - X_{\tilde{\mu}}\|_{L^2(\rho_0)}.
\end{equation}
Here the first inequality follows from the definition of $W_2$  and the characterization of $X^*$ and $\tilde X^*$ given in Lemma~\ref{lem:lift}.  The second is due to the fact that the $L^2(\rho_0)$ projection onto a convex set, {which in our case, is the set in \eqref{eq:constraintX},} is non-expansive. Since $X_\mu$ and $X_{\tilde{\mu}}$ are arbitrary and $\rho_0$ is absolutely continuous, minimizing over $X_\mu$ and $X_{\tilde{\mu}}$, by the same arguments as above, we obtain the {$W_2$-Lipschitz stability of the backward-projection}:

\begin{lemma}\label{lem:nonexpansive} Given $\nu \in \mc{P}_2(\mathbb{R}^d)$, for any $\mu, \tilde{\mu}\in\mc{P}_2(\mathbb{R}^d)$,
\[
W_2(\mathscr{P}_{\preceq \nu}^{cx}(\mu),\mathscr{P}_{\preceq \nu}^{cx}(\tilde{\mu}))\leq W_2(\mu,\tilde{\mu})\,.
\]
\end{lemma}
{(We note that this result has been independently established in the recent work \cite[Proposition 2.2]{alfonsi2025wasserstein} by a similar method.)}

As a direct consequence of equation \eqref{eq:stabproj} and Theorem \ref{th:projextra} (point (1)), we also obtain a $W_2$-Lipschitz stability result for the metric extrapolation:
\begin{proposition}\label{prop:extralip} Given  $\nu_0 \in \mc{P}_2(\mathbb{R}^d)$, for any $\nu_1, \tilde{\nu}_1\in\mc{P}_2(\mathbb{R}^d)$ and $t>1$,
\begin{equation}\label{eq:extralip}
W_2(\mathscr{E}^t(\nu_0,\nu_1),\mathscr{E}^t(\nu_0,\tilde{\nu}_1)) \leq 2t W_2(\nu_1,\tilde{\nu}_1).
\end{equation}
\end{proposition}
\begin{proof} 
For $\theta = (t-1)/t$, 
let $\mu = D^{1/\theta}_\# \nu_1$ and $ \tilde{\mu} = D^{1/\theta}_\# \tilde{\nu}_1$. By Theorem \ref{th:projextra} (point (1)), there exist $\theta$-strongly convex potentials $u$ and $\tilde{u}$ such that $\mathscr{P}^{cx}_{\preceq \nu_0}(\mu) = \nabla u^*_\# \nu_1$ and $\mathscr{P}^{cx}_{\preceq \nu_0}(\tilde{\mu}) = \nabla \tilde{u}^*_\# \tilde{\nu}_1$. 
Let $X_{\nu_1}, X_{\tilde{\nu}_1}\in L^2(\rho_0)$ be arbitrary maps such that $(X_{\nu_1})_\# \rho_0 = \nu_1$ and $(X_{\tilde{\nu}_1})_\# \rho_0 = \tilde{\nu}_1$. 
{Define
$\hbox{$X_\mu \coloneqq D^{1/\theta} \circ X_{\nu_1}$ and $X_{\tilde{\mu}} \coloneqq D^{1/\theta} \circ X_{\tilde{\nu}_1}$.}$
 Notice that  $T\coloneqq \nabla u^* \circ D^{\theta}$ and $\tilde{T} \coloneqq \nabla \tilde{u}^* \circ D^\theta$ are optimal transport maps  from $\mu$ to $\mathscr{P}^{cx}_{\preceq \nu_0}(\mu)$ and from $\tilde{\mu}$ to $\mathscr{P}^{cx}_{\preceq \nu_0}(\tilde{\mu})$, respectively (note in particular that $T$ and $\tilde{T}$ are the gradients of convex functions and they are 1-Lipschitz since $u$ and $\tilde{u}$ are} $\theta$-strongly convex).
Then, Lemma \ref{lem:lift} and equation \eqref{eq:stabproj} (the second inequality) imply that
\begin{align}\label{eq:extrastab1}
\|\nabla u^*\circ X_{\nu_1} - \nabla \tilde{u}^*\circ X_{\tilde{\nu}_1}\|_{L^2(\rho_0)}&  = \| T\circ X_\mu - \tilde{T}\circ X_{\tilde{\mu}} \|_{L^2(\rho_0)} \\\nonumber
& \le  \|  X_\mu -  X_{\tilde{\mu}} \|_{L^2(\rho_0)}
 = \frac{1}{\theta} \|X_{\nu_1} - X_{\tilde{\nu}_1}\|_{L^2(\rho_0)}\,.
\end{align}
Moreover, \[(tX_{\nu_1} + (1-t) \nabla u^*\circ X_{\nu_1},tX_{\tilde{\nu}_1} + (1-t)  \nabla \tilde{u}^*\circ X_{\tilde{\nu}_1})_\# \rho_0 \in \Pi(\mathscr{E}^t(\nu_0,\nu_1), \mathscr{E}^t(\nu_0,\tilde{\nu}_1))\,,\]
where $\Pi(\rho,\tilde{\rho})\in\mc{P}_2(\mathbb{R}^d\times \mathbb{R}^d)$ denotes the set of couplings with given marginals $\rho\in\mc{P}_2(\mathbb{R}^d)$ and $\tilde{\rho}\in\mc{P}_2(\mathbb{R}^d)$. Multiplying both sides of \eqref{eq:extrastab1} by $(t-1)$ (recalling $\theta = (t-1)/t$) and adding $t\|X_{\nu_1} - X_{\tilde{\nu}_1}\|_{L^2(\rho_0)}$ we get, by triangle inequality,
\[
{W_2(\mathscr{E}^t(\nu_0,\nu_1),\mathscr{E}^t(\nu_0,\tilde{\nu}_1))}\leq  2  t \|X_{\nu_1} - X_{\tilde{\nu}_1}\|_{L^2(\rho_0)}\,.
\]
Optimizing over $X_{\nu_1}$ and $ X_{\tilde{\nu}_1}$ we obtain the result. 
\end{proof}

\begin{remark}\label{rem:improve} Proposition \ref{prop:extralip} improves a H\"older stability estimate proved in \cite[Lemma 2.3]{Gallouet2025} in the second argument of the extrapolation operator $\mathscr{E}^t$. Note however that the Lipschitz constant in \eqref{eq:extralip} is  not sharp at least for $t\rightarrow 1^+$, since $\mathscr{E}^t(\nu_0,\nu_1) \rightarrow \nu_1$ as $t\rightarrow 1^+$ \cite{Gallouet2025}. 
\end{remark}

\begin{remark}
One may wonder whether the same strategy of Lemma~\ref{lem:nonexpansive} for the forward projection case would give a similar estimate. 
However, in  the forward case, the corresponding set 
\[
\{ X\in L^2(\rho_0) \,:\, \mu \preceq X_\# \rho_0\}
\]
is not convex with respect to the affine structure on $L^2(\rho_0)$, therefore, the second inequality in \eqref{eq:stabproj} does not follow. 
For example, take $X$ and $\rho_0$ symmetric with respect to 0 and so that $X_\#\rho_0$ and $(-X)_\# \rho_0$ are both equal to $\mu$ and absolutely continuous. If we interpolate linearly between $X$ and $-X$ we get $((X-X)/2)_\# \rho_0 = \delta_0$, and this does not dominate any absolutely continuous measure. Also, we note that  geodesic convexity (or even generalized geodesic convexity) also fails; in particular, see \cite[Section 8.2]{kim2024backward}.
\end{remark}

\section{Stability of  projections in convex order in one dimension}

In this section, we show how Theorem \ref{th:projextra} can be used to establish the Lipschitz continuity of the backward and forward projections in the one-dimensional case. This result was already established in  \cite{jourdain2023lipschitz}, but the proof we provide here is different and can be generalized to tackle the case $d>1$ as shown in the next sections.  Here, our approach is based on a reformulation of the extrapolation problem in terms of quantile functions. We remark that a related explicit formula for the solutions of the projection problems in terms of quantile functions was also established in \cite{alfonsi2020sampling}.

Given $\mu \in \mathcal{P}_2(\mathbb{R})$, let $F_\mu: \mathbb{R}\rightarrow [0,1]$ be its cumulative distribution function, defined by
\[
F_\mu(x) = \mu((-\infty,x])
\]
for all $x\in \mathbb{R}$. We denote by $X_\mu:(0,1) \rightarrow \mathbb{R}$ the pseudo-inverse of $F_\mu$,  which is also referred to as the quantile function of $\mu$. This is defined by \[X_\mu(a) \coloneqq \inf \{ x \;:\: F_\mu(x)>a\}\] for all $a\in(0,1)$. With these definitions, for any measures
$\mu, \nu \in \mathcal{P}_2(\mathbb{R})$,
\begin{equation}\label{eq:w2oned}
W_2^2(\mu,\nu) = \int_0^1 |X_\mu(a)-X_\nu(a)|^2\ed a\,.
\end{equation}

\subsection{Stability of the metric extrapolation in one dimension} \label{sec:extrastab} We start by recalling the Lipschitz continuity of the metric extrapolation operator for $d=1$, already proven in \cite{tornabene2024generalized}, for example. 
This follows from the representation of $W_2$ as an $L^2$ distance in \eqref{eq:w2oned}, which implies an explicit characterization of the solutions to \eqref{eq:extrapolation}.

Let us denote  by $\mc{M}$ the set of nondecreasing functions in $L^2(0,1)$.  Replacing \eqref{eq:w2oned} into \eqref{eq:extrapolation} and rearranging the squares, we obtain that the metric extrapolation at time $t>1$ of the geodesic connecting two  measures $\nu_0,\nu_1 \in \mathcal{P}_2(\mathbb{R})$, is given by
\[
\mathscr{E}^t(\nu_0,\nu_1) = Y^t_\# \mathrm{Leb}|_{[0,1]} \,, \quad  Y^t\coloneqq P_{\mathcal{M}}( t X_{\nu_1} + (1-t) X_{\nu_0})\,,
\]
where $P_{\mc M}$ denotes the $L^2$ orthogonal projection onto $\mc{M}$.  
This implies the following stability bound 
\begin{align}\label{eq:1d-stablility-extra}                 W_2(\mathscr{E}^t(\nu_0,\nu_1),\mathscr{E}^t(\tilde{\nu}_0,\tilde{\nu}_1))& \leq \| t(X_{\nu_1} - X_{\tilde{\nu}_1}) + (1-t)(X_{\nu_0} - X_{\tilde{\nu}_0})\|_{L^2(0,1)}\\\nonumber &\leq   (t-1) W_2(\nu_0,\tilde{\nu}_0) + tW_2(\nu_1,\tilde{\nu}_1)\,.
\end{align}

{Note that $ s\mapsto Y^s_\# \mathrm{Leb}|_{[0,1]}$ for $s\in[0,1]$ coincides with the $W_2$-geodoesic from $\nu_0$ to $\nu_1$. Using this fact, one can easily adapt the argument above to obtain a similar stability estimate for the geodesic interpolation. Specifically, for any minimizing geodesics $\gamma,\tilde{\gamma}:[0,1] \rightarrow \mc{P}_2(\mathbb{R})$,
\begin{equation}\label{eq:interpolation}
W_2(\gamma(s),\tilde{\gamma}(s))\leq (1-s) W_2(\gamma(0), \tilde{\gamma}(0)) + s W_2(\gamma(1),\tilde{\gamma}(1)),
\end{equation}
for all $s\in[0,1]$.
}

\subsection{Stability of the backward projection in one dimension}\label{sec:stabbackone} Combining Theorem \ref{th:projextra} (point (1)) with the stability bound above for the extrapolation in one dimension, and denoting $\theta = (t-1)/t$ as in the statement of the theorem, we get
\begin{equation}\label{eq:stability1d}
\begin{aligned}
    W_2(\mathscr{P}^{cx}_{\preceq \nu} (\mu), \mathscr{P}^{cx}_{\preceq \tilde{\nu}} (\tilde{\mu})) &\leq  \frac{W_2(\mathscr{E}^t(\nu, D^\theta_\#\mu), \mathscr{E}^t(\tilde{\nu}, D^\theta_\#\tilde{\mu}))}{t-1} + \frac{t}{t-1}W_2(D^\theta_\#\mu,D^\theta_\#\tilde{\mu}) \\
& \leq  W_2(\nu,\tilde{\nu}) + \frac{2t}{t-1} W_2(D^\theta _\#\mu,D^\theta_\#\tilde{\mu}) \\
&= W_2(\nu,\tilde{\nu}) + 2 W_2(\mu,\tilde{\mu})\,.
\end{aligned}
\end{equation}
Indeed,
it is easy to see from Theorem \ref{th:projextra} (and specifically by a time rescaling of $\gamma$ in point (1)) that  $\mathscr{P}^{cx}_{\preceq \nu} (\mu)$ is the metric extrapolation at time $t/(t-1)$ of the geodesic starting at $\mathscr{E}^t(\nu, D^\theta_\#\mu)$ at time $0$ and passing through  $D^\theta_\#\mu$ at time $1$.
Since a similar interpretation holds for $\mathscr{P}^{cx}_{\preceq \tilde{\nu}} (\tilde{\mu})$ we obtain the first line in \eqref{eq:stability1d} by applying \eqref{eq:1d-stablility-extra} (with $t$ replaced by $t/(t-1)$). The second line also comes from stability bound \eqref{eq:1d-stablility-extra} applied to the term $W_2(\mathscr{E}^t(\nu, D^\theta_\#\mu), \mathscr{E}^t(\tilde{\nu}, D^\theta_\#\tilde{\mu}))$, and the third from the equality \[
 W_2(D^\theta _\#\mu,D^\theta_\#\tilde{\mu}) =\theta  W_2(\mu,\tilde{\mu}) \,, \quad \text{with } \quad \theta = \frac{t-1}{t}\,.
\]

\begin{figure}
    \begin{tikzpicture}[scale=1]

\begin{scope}

  \node (a) at (0,.3) {} ;

  \node (a0) at (0,1) {};

  \node (b) at (1,1) {} ; 
  \node (c) at (3,1) {};

  \node[gray,below] at (-1.4,-.2) {time};

\draw[->, thick, red]
  (c) to (a0);

  \fill[red] (a0) circle (2pt) node[left] {$\mathscr{P}^{cx}_{\preceq \nu} (\mu)$};

  \fill (a) circle (2pt) node[left] {$\nu$};

  \fill (b) circle (2pt) node[above] {$D^\theta_\# \mu$}; 
  \fill (c) circle (2pt) node[below right] {$\mathscr{E}^t(\nu, D^\theta_\# \mu)$};

\draw[->, thick]
  (a) to[out=38, in=150] (c);

\draw[dashed,gray]
   (a0) to (0,-.2) node[below] {$0$};

\draw[dashed, gray]
(b) to (1,-.2) node[below] {$1$};

\draw[dashed, gray]
(c) to (3,-.2) node[below] {$t$};
\end{scope}

\begin{scope}[xshift=7cm]

\node (a) at (0,.3) {} ;

\node (a0) at (1,.53) {};

\node (b) at (1,1) {} ; 
\node (c) at (3,1) {};

\draw[->, thick, red]
(c) to (a);

\fill[red] (a0) circle (2pt) node [below right] {$\mathscr{P}^{cx}_{\mu \preceq } (\nu)$};

\fill (a) circle (2pt) node[left] {$D^{1/\theta}_\#\nu$};

\fill (b) circle (2pt) node[above] {$\mu$}; 
\fill (c) circle (2pt) node[below right] {$ \mathscr{E}^t(D^{1/\theta}_\#\nu,  \mu)$};

\draw[->, thick]
(a) to[out=38, in=150] (c);

\draw[dashed,gray]
(a) to (0,-.2) node[below] {$0$};

\draw[dashed, gray]
(b) to (1,-.2) node[below] {$1$};

\draw[dashed, gray]
(c) to (3,-.2) node[below] {$t$};
\end{scope}

\end{tikzpicture} \vspace{-3em}
\caption{In Section \ref{sec:stabbackone}, the backward projection is viewed as the composition of two extrapolations (black and red arrows). Similarly, in Sections \ref{sec:stabforwardone} and \ref{sec:forstabd}, the forward projection is viewed as the composition of an extrapolation and a Wasserstein barycenter. }\label{fig:illustration}
\end{figure}

 An illustration of the main idea behind the argument can be found in Figure \ref{fig:illustration} (left).
Note that the estimate \eqref{eq:stability1d} coincides with the stability result proved in \cite{jourdain2023lipschitz}.  {When $\mu=\tilde \mu$ this gives the $1$-Lipschitz dependence for $\nu \to \mathscr{P}^{cx}_{\preceq \nu} (\mu)$. Therefore,} combining this with Lemma \ref{lem:nonexpansive}, by triangle inequality, we obtain
\begin{proposition}
    For any $\mu, \tilde{\mu},\nu,\tilde{\nu}\in\mc{P}_2(\mathbb{R})$,
\[
W_2(\mathscr{P}_{\preceq \nu}^{cx}(\mu),\mathscr{P}_{\preceq \tilde{\nu}}^{cx}(\tilde{\mu}))\leq W_2(\mu,\tilde{\mu})+ W_2(\nu,\tilde{\nu})\,.
\]
\end{proposition}
{(This consequence of Lemma \ref{lem:nonexpansive} is also contained in \cite[Corollary 2.3]{alfonsi2025wasserstein}.)}
\subsection{Stability of the forward projection in one dimension} \label{sec:stabforwardone} The argument for the forward projection is very similar; see again Figure \ref{fig:illustration} (right) for a graphical illustration.
By Theorem \ref{th:projextra} (point (2)) and a time rescaling, if $\nu$ is absolutely continuous, $\mathscr{P}_{\mu \preceq }^{cx}(\nu) = {\gamma}(1/t)$ for a minimizing geodesic ${\gamma}:[0,1]\rightarrow \mc{P}_2(\mathbb{R}^d)$ satisfying ${\gamma}(0)= D^{1/\theta}_\#\nu$ and ${\gamma}(1) = \mathscr{E}^t(D^{1/\theta}_\#\nu, \mu)$. Hence using equation \eqref{eq:interpolation} with $s= 1/t$, we obtain
\[
\begin{aligned}
W_2(\mathscr{P}_{\mu \preceq }^{cx}(\nu), \mathscr{P}_{\tilde{\mu} \preceq }^{cx}(\tilde{\nu}))&\leq \left(1 - \frac{1}{t}\right) W_2(D^{1/\theta}_\#\nu,D^{1/\theta}_\#\tilde{\nu}) + \frac{1}{t}W_2(\mathscr{E}^t(D^{1/\theta}_\#\nu,  \mu),\mathscr{E}^t(D^{1/\theta}_\#\tilde{\nu}, \tilde{\mu}))\\
& \leq 2\frac{t-1}{t} W_2(D^{1/\theta}_\#\nu,D^{1/\theta}_\#\tilde{\nu}) +  W_2(\mu, \tilde{\mu})\,.
\end{aligned}
\]
Using the equality $ W_2(D^\theta_\# \nu, D^\theta_\#\tilde{\nu}) = W_2(\nu,\tilde{\nu})/\theta$, and $\theta = \frac{t-1}{t}$,  we deduce that

\begin{proposition}\label{prop:lipforward1d}
    For any $\mu, \tilde{\mu}\in\mc{P}_2(\mathbb{R})$ and $\nu,\tilde{\nu}\in\mc{P}^{ac}_2(\mathbb{R})$,
\[
    W_2(\mathscr{P}_{\mu \preceq }^{cx}(\nu),\mathscr{P}_{\tilde{\mu} \preceq }^{cx}(\tilde{\nu}))\leq 2 W_2(\nu,\tilde{\nu}) +  W_2( \mu, \tilde{\mu})\,.
\]
\end{proposition}

Finally, we note that even if our approach is limited to the $W_2$ setting, the Lipschitz bounds in Proposition \ref{prop:lipforward1d} and equation \eqref{eq:stability1d} actually hold with the same constants for all $W_p$ distances with $p \in [1,\infty)$ as shown in \cite[Theorem 1.1]{jourdain2023lipschitz}.

\section{Stability of projection in convex order in arbitrary dimensions}

\subsection{Stability of the backward projection in arbitrary dimensions}

In this section we prove a stability result on the projection map $\mathscr{P}^{cx}_{\preceq \nu}(\mu)$,
in arbitrary  dimensions. The proof still relies on Theorem \ref{th:projextra} but it is different from the one-dimensional case.
Specifically, we will slightly adapt the argument used in \cite{Gallouet2025} to prove the stability of the metric extrapolation. This relies on two basic estimates on the extrapolation problem. 
First, since the functional
\[
\mc{G}_t(\nu_0,\nu_1; \mu) \coloneqq t \frac{W^2(\nu_1,\mu)}{2} - (t-1) \frac{W^2(\nu_0,\mu)}{2}
\]
for each $t>1$, is $1$-strongly convex along generalized geodesics with base $\nu_1$ (see, e.g., \cite{ambrosio2005gradient} for the definition of generalized geodesic), one can deduce that  for any $\rho \in \mc{P}_2(\mathbb{R}^d)$, if there exist optimal transport maps $T$ and $T^\rho$, from $\nu_1$ to $\mathscr{E}^t(\nu_0,\nu_1)$ and from $\nu_1$ to $\rho$, respectively, then
\begin{equation}\label{eq:strongc}
\frac{1}{2} \int_{\mathbb{R}^d} |T(x)-T^\rho(x)|^2\ed \nu_1(x) + \mc{G}_t(\nu_0,\nu_1; \mathscr{E}^t(\nu_0,\nu_1)) \leq \mc{G}_t(\nu_0,\nu_1; \rho) \,.
\end{equation}
Furthermore, from the so-called dissipative property of $\mathscr{E}^t(\nu_0,\nu_1)$ (see Definition 1.1 in \cite{gallouet2024geodesic}), we have 
\begin{align}\label{eq:extrapolation-ineq}
W_2(\nu_0, \mathscr{E}^t(\nu_0,\nu_1)) \leq t W_2(\nu_0,\nu_1) \,, \quad W_2(\nu_1, \mathscr{E}^t(\nu_0,\nu_1)) \leq (t-1) W_2(\nu_0,\nu_1)\,;
\end{align}
see \cite{gallouet2024geodesic} for more details.

\begin{proposition}\label{prop:backwardstab}
For any $\mu, \tilde{\mu}, \nu,\tilde{\nu}\in\mc{P}_2(\mathbb{R}^d)$ with $d \ge 1$,
\[
W_2(\mathscr{P}_{\preceq \nu}^{cx}(\mu),\mathscr{P}_{\preceq \tilde{\nu}}^{cx}(\tilde{\mu}))\leq W_2(\mu,\tilde{\mu}) + \sqrt{(W_2(\mu,\nu) + W_2(\mu,\tilde{\nu}))W_2(\nu,\tilde{\nu})}\,.
\]
\end{proposition}
{(We note that, up to minor modifications, this estimate coincides with the one recently derived in \cite[Proposition 2.9]{alfonsi2025wasserstein} but using a different method.)}
\begin{proof} Let us set $\nu_0 =\nu$, $\tilde{\nu}_0=\tilde{\nu}$ and $\nu_1 = D^\theta_\# \mu$ for a fixed $t>1$ and $\theta=(t-1)/t$, and let us use, for simplicity, $\mc{G}$ and $\tilde{\mc{G}}$ to denote $\mc{G}_t(\nu_0,\nu_1;\cdot)$ and $\mc{G}_t(\tilde{\nu}_0,{\nu}_1;\cdot)$, respectively. By Theorem \ref{th:projextra} (point (1)), 
we can write
\[
 \mathscr{E}^t(\nu_0,\nu_1) = (t\mathrm{Id} +(1-t) \nabla u^*)_\# \nu_1\,, \quad  \mathscr{E}^t(\tilde{\nu}_0,\nu_1) = (t\mathrm{Id}+(1-t)  \nabla \tilde{u}^* )_\# \nu_1\,
\]
for some $\theta$-strongly convex functions $u$ and $\tilde{u}$ such that 
\begin{equation}\label{eq:backward}
    \mathscr{P}^{cx}_{\preceq \nu_0}(\mu) = \nabla u^*_\# \nu_1\,, \quad \mathscr{P}^{cx}_{\preceq \tilde{\nu}_0}(\mu) = \nabla \tilde{u}^*_\# \nu_1\,.
\end{equation}
As $(\nabla u^*,\nabla \tilde u^*)_\# \nu_1$ is a coupling between $(\nabla u^*)_\# \nu_1$ and $(\nabla \tilde{u}^*)_\# \nu_1$, we obtain
\begin{align}\label{eq:Wass-u-tilde-u}
W_2^2(\mathscr{P}^{cx}_{\preceq \nu_0}(\mu),\mathscr{P}^{cx}_{\preceq \tilde{\nu}_0}(\mu))   \le \iint_{\mathbb{R}^d\times \mathbb{R}^d} |\nabla u^*(x)-\nabla \tilde{u}^*(x)|^2\ed \nu_1(x).
\end{align}

Now, let
\[
    \nu_t \coloneqq \mathscr{E}^t(\nu_0,\nu_1), \quad \tilde{\nu}_t \coloneqq \mathscr{E}^t(\tilde{\nu}_0,\nu_1)\,.
\]
Setting $\rho = \tilde{\nu}_t$ 
 in \eqref{eq:strongc}, due to Theorem \ref{th:projextra} (point (1)),   we can choose 
\[T = t\mathrm{Id} +(1-t) \nabla u^*\quad \text{and} \quad T^\rho = t \mathrm{Id} +(1-t)\nabla\tilde{u}^*\,. \] 
Then \eqref{eq:strongc}   yields
\begin{equation}\label{eq:firstestimate}
\frac{(t-1)^2}{2} \iint_{\mathbb{R}^d\times \mathbb{R}^d} |\nabla u^*(x)-\nabla \tilde{u}^*(x)|^2\ed \nu_1(x) \leq \mc{G}(\tilde{\nu}_t) - \mc{G}(\nu_t)\,,
\end{equation}
Similarly, setting $\rho=\nu_t$ in \eqref{eq:strongc} written for $\tilde{\mc{G}}$, and summing up the result with equation \eqref{eq:firstestimate} we obtain
\begin{equation}\label{eq:secondestimate}
    (t-1)^2 \iint_{\mathbb{R}^d\times \mathbb{R}^d} |\nabla u^*(x)-\nabla \tilde{u}^*(x)|^2\ed \nu_1(x) \leq \mc{G}(\tilde{\nu}_t) - \mc{G}(\nu_t) +\tilde{\mc{G}}(\nu_t) - \tilde{\mc{G}}(\tilde{\nu}_t)\,.
\end{equation}
The right-hand side is equal to
\begin{equation}\label{eq:rhsstab}
    \frac{t-1}{2}(W^2_2(\tilde{\nu_t},\tilde{\nu_0})-W_2^2(\tilde{\nu}_t,\nu_0) )  +\frac{t-1}{2}(W^2_2(\nu_t,\nu_0)-W^2_2(\nu_t,\tilde{\nu_0}))\,.
\end{equation}
 Applying the triangle inequality (twice) to the first term, we obtain
\[
\begin{aligned}
    W^2_2(\tilde{\nu_t},\tilde{\nu_0})-W_2^2(\tilde{\nu}_t,\nu_0) & = (W_2(\tilde{\nu_t},\tilde{\nu_0})-W_2(\tilde{\nu}_t,\nu_0))(W_2(\tilde{\nu_t},\tilde{\nu_0})+W_2(\tilde{\nu}_t,\nu_0))\\&\leq W_2(\nu_0,\tilde{\nu}_0) (W_2(\tilde{\nu_t},\tilde{\nu_0})+W_2(\tilde{\nu}_t, \nu_1) + W_2(\nu_1, \nu_0)).
\end{aligned}
\]
Then, using the dissipative property of the metric extrapolation \eqref{eq:extrapolation-ineq} on $W_2(\tilde{\nu}_t, \nu_1)$ and $W_2(\tilde{\nu_t},\tilde{\nu_0})$, we find 
\[
    W^2_2(\tilde{\nu_t},\tilde{\nu_0})-W_2^2(\tilde{\nu}_t,\nu_0) \leq W_2(\nu_0,\tilde{\nu}_0) ((2t-1) W_2(\nu_1,\tilde{\nu_0})+W_2({\nu}_1, \nu_0))\,.
\]
Proceeding in a similar way for the second term in \eqref{eq:rhsstab} we also obtain
\[
\begin{aligned}
   W^2_2(\nu_t,\nu_0)-W_2^2(\nu_t,\tilde{\nu_0})
    &\leq W_2(\nu_0,\tilde{\nu}_0) ((2t-1) W_2(\nu_1,\nu_0)+W_2({\nu}_1, \tilde{\nu}_0)).
\end{aligned}
\]

These  together with \eqref{eq:Wass-u-tilde-u} give
\begin{equation}\label{eq:thirdestimate}
    (t-1)^2W_2^2(\mathscr{P}^{cx}_{\preceq \nu_0}(\mu),\mathscr{P}^{cx}_{\preceq \tilde{\nu}_0}(\mu)) \leq t (t-1) (W_2(\nu_1,\nu_0) + W(\nu_1,\tilde{\nu}_0))W_2(\nu_0,\tilde{\nu}_0)\,,
\end{equation}
where we used \eqref{eq:backward} to bound from below the left-hand side of \eqref{eq:secondestimate} by the left-hand side of \eqref{eq:thirdestimate}. Recalling that $\nu_1 = D^\theta_\#\mu$, by triangle inequality,
\[
\begin{aligned}
    W_2(\nu_1,\nu_0) + W(\nu_1,\tilde{\nu}_0) 
    &\leq  W_2(\nu_0,\mu) +W_2(\tilde{\nu}_0,\mu)+ 2W_2(\mu,D^\theta_\#\mu)\\
    &= W_2(\nu_0,\mu) +W_2(\tilde{\nu}_0,\mu)+2(1-\theta) M_2(\mu)
\end{aligned}
\]
where $M_2^2(\mu) \coloneqq \int |x|^2\ed \mu(x)$.
Reinserting this into \eqref{eq:thirdestimate} and dividing the equation by $(t-1)^2$, we obtain
\[
W_2^2(\mathscr{P}_{\preceq \nu_0}(\mu),\mathscr{P}_{\preceq \tilde{\nu}_0}(\mu)) \leq \left[\frac{t}{t-1}(W_2(\mu,\nu_0) + W_2(\mu,\tilde{\nu}_0)) +\frac{2}{t-1} M_2(\mu)\right] W_2(\nu_0,\tilde{\nu}_0)\,.
\]
Finally, letting $t\rightarrow \infty$ we derive 
\[
W_2^2(\mathscr{P}_{\preceq \nu}^{cx}(\mu),\mathscr{P}_{\preceq \tilde{\nu}}^{cx}(\mu))\leq (W_2(\mu,\nu) + W_2(\mu,\tilde{\nu}))W_2(\nu,\tilde{\nu})\,.
\]
We conclude using the triangle inequality and Lemma \ref{lem:nonexpansive}.
\end{proof}

Note that by the same arguments, we also have:

\begin{lemma}[Lemma 2.3 in \cite{Gallouet2025}]\label{lem:extraholder}
For any $\nu_1,\nu_0,\tilde{\nu}_0\in\mc{P}_2(\mathbb{R}^d)$,
\[
\begin{aligned}
    W_2^2(\mathscr{E}^t(\nu_0,\nu_1),\mathscr{E}^t(\tilde{\nu}_0,\nu_1)) &\leq  t (t-1) (W_2(\nu_1,\nu_0) + W(\nu_1,\tilde{\nu}_0))W_2(\nu_0,\tilde{\nu}_0)\,,
\end{aligned}
\]
\end{lemma}

\subsection{Stability of the forward projection in arbitrary dimensions}\label{sec:forstabd} To prove the stability of the forward projection, we use the same strategy we used in the one-dimensional case. This will result in a worse estimate, since the available stability estimate for both geodesic interpolations and extrapolations are only H\"older. We start by recalling a stability result on geodesic interpolations adapted from \cite{carlier2024quantitative}. Define:
\begin{multline*}
\mc{P}^{reg}_2(\mathbb{R}^d) \coloneqq \{ \rho \in \mc{P}^{ac}_2(\mathbb{R}^d)\,, \, \mathrm{spt}(\rho) \text{ is compact and convex}, \\ \, \exists \,m, M>0~ \text{such that} ~m\leq \rho |_{\mathrm{spt}(\rho)}\leq M\}\,.
\end{multline*}

\begin{lemma}[adapted from Theorem 1.5 in \cite{carlier2024quantitative}]\label{lem:stabbary} Let $\nu_0 \in \mc{P}_2^{reg}(\mathbb{R}^d)$ and $\nu_1,\tilde{\nu}_0,\tilde{\nu}_1 \in \mc{P}_2(\mathbb{R}^d)$. Let $\gamma,\tilde{\gamma}:[0,1] \rightarrow \mc{P}_2(\mathbb{R}^d)$ be any minimizing geodesic such that $\gamma(0) =\nu_0$, $\gamma(1)=\nu_1$,  $\tilde{\gamma}(0) =\tilde{\nu}_0$, $\tilde{\gamma}(1)=\tilde{\nu}_1$. Then, there exists a constant $C_{d,\nu_0}>0$ only depending on $d$ and $\nu_0$ such that
\[
W_2(\gamma(s),\tilde{\gamma}(s)) \leq C_{d,\nu_0} \left( W_1(\nu_0,\tilde{\nu}_0)+ \frac{s}{1-s} W_1(\nu_1,\tilde{\nu}_1)\right)^{1/6} \, { \quad \text{ for all $s\in [0,1)$}}.
\]
\end{lemma}
\begin{remark}
The estimate remains valid under less restrictive assumptions on $\nu_0$, allowing for measures with non convex supports, and which we omit here for brevity; see \cite{carlier2024quantitative} for a detailed discussion. The stability result below still holds with minor modifications under these more general assumptions.
\end{remark}

\begin{proposition}\label{prop:forwardstab} Let $\nu \in \mc{P}^{reg}_2(\mathbb{R}^d)$, $\tilde{\nu} \in \mc{P}^{ac}_2(\mathbb{R}^d)$ and $\mu,\tilde{\mu} \in \mc{P}_2(\mathbb{R}^d)$. Then
\[
W_2(\mathscr{P}_{\mu \preceq }^{cx}(\nu),\mathscr{P}_{\tilde{\mu} \preceq }^{cx}(\tilde{\nu}))\leq C_{d,\nu} ( 2^{1/6} W_2(\mu,\tilde{\mu})^{1/6} + W_2(\nu,\tilde{\nu})^{1/6} + C W_2(\nu,\tilde{\nu})^{1/12})\,,
\]
where $C_{d,\nu}>0$ is as in Lemma \ref{lem:stabbary} and $C\leq  (W_2(\nu,\tilde{\mu}) +W_2(\tilde{\nu},\tilde{\mu}))^{1/12} $.
\end{proposition}
\begin{proof}
We treat separately the dependency of $\mathscr{P}^{cx}_{\mu \preceq}(\nu)$ on $\mu$ and $\nu$.
As for the dependency on $\mu$, 
 we proceed as in Section \ref{sec:stabforwardone}.  Fix $t>1$.  Assuming only $\nu \in \mc{P}_2^{{reg}}(\mathbb{R}^d)$, we have
by Theorem \ref{th:projextra} (point (2)) and a time rescaling,
 that $ D^\theta_\#\mathscr{P}_{\mu \preceq }^{cx}(\nu) = {\gamma}(1/t)$ for a minimizing geodesic ${\gamma}:[0,1]\rightarrow \mc{P}_2(\mathbb{R}^d)$ satisfying ${\gamma}(0)= \nu$ and ${\gamma}(1) = D^\theta_\#\mathscr{E}^t(D^{1/\theta}_\#\nu, \mu)$. Hence, applying Lemma \ref{lem:stabbary} with $s=1/t$, we get 
\begin{equation}\label{eq:forwardestimate}
\begin{aligned}
W_2(D^\theta_\#\mathscr{P}_{\mu \preceq }^{cx}(\nu), D^\theta_\#\mathscr{P}_{\tilde{\mu} \preceq }^{cx}(\nu))&\leq C_{d,\nu} \left( \frac{W_1(\mathscr{E}^t( \nu, D^\theta_\#\mu),\mathscr{E}^t({\nu},  D^\theta_\#\tilde{\mu})}{t-1}\right)^{1/6}\,,
\end{aligned}
\end{equation} 
where we used the equality $D^\theta_\#\mathscr{E}^t(D^{1/\theta}_\#\nu, \mu)= \mathscr{E}^t(\nu,D^\theta_\# \mu)$. Note that compared to the statement in Theorem \ref{th:projextra} we have rescaled the geodesic in space by $D^\theta_\#$ so that $\gamma(0)$ does not depend on $\theta$, and we can use Lemma \ref{lem:stabbary} with a constant $C_{d,\nu}$ independent of $\theta$. 

Since $W_1\leq W_2$, we can use the $W_2$ distance on the right-hand side.
Hence, dividing both sides by $\theta$ using the Lipschitz estimate in Proposition \ref{prop:extralip}, we obtain
\begin{align*}
W_2(\mathscr{P}_{\mu \preceq }^{cx}(\nu),\mathscr{P}_{\tilde{\mu} \preceq }^{cx}(\nu)) \leq \frac{2^{1/6}C_{d,\nu}}{\theta} W_2(\mu,\tilde{\mu})^{1/6}\,.
\end{align*}
Letting $t\rightarrow \infty$ we obtain
\begin{align}\label{eq:forward-mu-tildemu}
W_2(\mathscr{P}_{\mu \preceq }^{cx}(\nu),\mathscr{P}_{\tilde{\mu} \preceq }^{cx}(\nu)) \leq 2^{1/6}C_{d,\nu} W_2(\mu,\tilde{\mu})^{1/6}\,.
\end{align}

For the dependency on $\nu$, we proceed similarly.
By Lemma \ref{lem:stabbary}, where again we replace $W_1$ with $W_2$,
\[
W_2(D^\theta_\#\mathscr{P}_{\mu \preceq }^{cx}(\nu), D^\theta_\#\mathscr{P}_{\mu \preceq }^{cx}(\tilde{\nu}))\leq C_{d,\nu} \left( W_2(\nu,\tilde{\nu})+ \frac{W_2(\mathscr{E}^t(\nu,  D^\theta_\#\mu),\mathscr{E}^t(\tilde{\nu},  D^\theta_\#\mu)}{t-1}\right)^{1/6}.
\]
Now, by Lemma \ref{lem:extraholder}  and the triangle inequality,  \[W_2(\mathscr{E}^t(\nu, D^\theta_\# \mu),\mathscr{E}^t( \tilde{\nu},  D^\theta_\#\mu))\leq C_{t,\nu,\tilde{\nu},\mu} W_2(\nu,\tilde{\nu})^{1/2},\] where
\[
C_{t,\nu,\tilde{\nu},\mu}= (t(t-1)(W_2(\nu,\mu) +W_2(\tilde{\nu},\mu)))^{1/2} +(2(t-1)M_2(\mu))^{1/2}\,.
\]
Therefore,
\[
W_2(\mathscr{P}_{\mu \preceq }^{cx}(\nu),\mathscr{P}_{\mu \preceq }^{cx}(\tilde{\nu}))\leq \frac{C_{d,\nu}}{\theta} \left( W_2(\nu,\tilde{\nu})+ C_{t,\nu,\tilde{\nu},\mu} \frac{W_2(\nu, \tilde{\nu})^{1/2}}{(t-1)}\right)^{1/6}\,.
\]
Letting $t\rightarrow \infty$, we recover
\[
W_2(\mathscr{P}_{\mu \preceq }^{cx}(\nu),\mathscr{P}_{\mu \preceq }^{cx}(\tilde{\nu}))\leq C_{d,\nu} \left( W_2(\nu, \tilde{\nu})^{1/6} +(W_2(\nu,\mu) +W_2(\tilde{\nu},\mu))^{1/12}  W_2(\nu, \tilde{\nu})^{1/12}\right)\,.
\]
We combine this with \eqref{eq:forward-mu-tildemu}
and conclude by the triangle inequality.
\end{proof}

\bibliographystyle{plain}
\bibliography{reference}   
\end{document}